\documentclass[11pt]{amsart}
\usepackage[headings]{fullpage}

\usepackage{amssymb,amsmath,amstext,amsthm,mathtools,caption}
\usepackage{enumerate,url,comment,tikz,tikz-cd}
\usepackage{cmbright}

\usepackage[colorlinks=true,urlcolor=cyan,linkcolor=blue,citecolor=magenta]{hyperref}

\usepackage{xpatch}
\xpatchcmd{\qed}{\hfill}{}{}{}
\xpatchcmd{\qed}{\quad}{\qquad}{}{}

\theoremstyle{theorem}
\newtheorem{theorem}{Theorem}
\newtheorem{proposition}[theorem]{Proposition}
\newtheorem{lemma}[theorem]{Lemma}
\newtheorem{corollary}[theorem]{Corollary}
\newtheorem{conjecture}[theorem]{Conjecture}
\newtheorem{problem}[theorem]{Problem}

\theoremstyle{definition}
\newtheorem{definition}[theorem]{Definition}

\newcommand{\conv}{\mathrm{conv}}
\newcommand{\vertices}{\mathrm{Vert}}

\newcommand{\set}[2]{\left\{ #1\,\middle|\, #2\right\}}

\newcommand{\gc}{g^c}
\newcommand{\hc}{h^c}
\newcommand{\hsc}{h^{sc}}

\newcommand{\dhalf}[1]{\left\lfloor #1/2\right\rfloor}
\newcommand{\duhalf}[1]{\left\lceil #1/2\right\rceil}
\newcommand{\vspan}{\mathrm{Span}}
\newcommand{\R}{\mathbb{R}}

\newcommand{\pn}{{\sf PN}}
\newcommand{\pr}{{\sf PR}}

\def\mchoose#1#2{\ensuremath{\left(\kern-.3em\left(\genfrac{}{}{0pt}{}{#1}{#2}\right)\kern-.3em\right)}}

\title{On the realization space of the cube}

\author{Karim Adiprasito}
\address{Department of Mathematics, University of Copenhagen, Copenhagen \and\newline Einstein Institute of Mathematics, The Hebrew University of Jerusalem}
\email{adiprasito@math.huji.ac.il}

\author{Daniel Kalmanovich}
\address{Einstein Institute of Mathematics, The Hebrew University of Jerusalem}
\email{daniel.kalmanovich@gmail.com}

\author{Eran Nevo}
\address{Einstein Institute of Mathematics, The Hebrew University of Jerusalem}
\email{nevo.eran@gmail.com}

\begin{document}

\begin{abstract}
	We consider the realization space of the $d$-dimensional cube, and show that
	any two realizations are connected by a finite sequence of projective transformations and normal transformations. We use this fact to define an analog of the connected sum construction for cubical $d$-polytopes, and apply this construction to certain cubical $d$-polytopes to conclude that the rays spanned by $f$-vectors of cubical $d$-polytopes are dense in Adin's cone. The connectivity result on cubes extends to any product of simplices, and further, it shows the respective realization spaces are  contractible.
\end{abstract}

\maketitle

\section{Introduction}
Perhaps the most natural transformations on polytopes that preserve the combinatorial type, namely the facial structure, are projective transformations
and normal transformations.
Loosely speaking, the former are given by perspective transformation from one hyperplane where the polytope lies to another hyperplane,
while the latter are given by scaling the outer normal vectors to facets so that facets do not degenerate. While the former are connected to the projective linear group acting on vector spaces, the latter is connected to the Chow cohomology of toric varieties, and in particular inherits an algebra structure via the Minkowski sum~\cite{McMu93}. (By \emph{polytope} we always mean a convex polytope.)

The simplex, of any fixed dimension, is \emph{projectively unique}, namely, any simplex can be continuously transformed to any other simplex of same dimension by a homotopy of projective transformations. Thus, any two \emph{simplicial} polytopes, after applying an appropriate projective transformation to one of them, can be glued along a common facet whose hyperplane separates them, to produce again a convex polytope. This realizes the connected sum operation geometrically.

However, the $d$-cube is not projectively unique for $d\ge 3$; this can be seen even by dimension count: the realization space of the (combinatorial) $d$-cube has dimension larger then the dimension of the space of projective transformations. Indeed, the group of projective transformations on $\mathbb{R}^d$ is of dimension $d(d+2)$, while the realization space of the $d$-cube has dimension $2d^2$.

In particular, we can not realize the connected sum operation geometrically for cubical $d$-polytopes, $d\ge 4$.

We enlarge the set of transformations by adding normal transformations to the generating set. While the first author mentioned this theorem in passing, assuming it had to be known, it was to our surprise that the following results appear to be new, even the qualitative assertion in $(a)$.

\begin{theorem}[Cubes are normal-projectively unique]\label{thm:eq-cubes}
	Fix a dimension $d$.
	\begin{enumerate}[(a)]
		\item
		For any two realization of the $d$-cube, one can be obtained from the other by a composition of finitely many transformations, each is either projective or normal. In fact, $8d$ of them suffice.
		
		
		
		\item The constructed algorithm transforms cubes continuously to the standard cube. In particular, we obtain a deformation retraction to a point. The realization space of cubes is contractible.
	\end{enumerate}
\end{theorem}

Let us stress that we stay entirely inside the space of cubes. Every transformation takes us from one cube to another; not one of the projective transformations results in an unbounded polytope.

As a corollary of the quantitative assertion in $(a)$, we obtain a cubical analog of the connected sum construction, at a small price.

\begin{theorem}\label{thm:cubical.conn.sum}
	\begin{enumerate}[(a)]
		\item (Bounded towers) For any two realizations $C_1$ and $C_2$ of the $(d-1)$-cube, there exists a cubical $d$-polytope $C$ made of $m$ ($m\le 4d$) $d$-cubes stacked one on top of the other, such that $C_1$ and $C_2$ are projectively equivalent to its bottom and top facets, resp. Call $C$ a \emph{$d$-tower of $m$ cubes}.
		
		\item (Cubical connected sum) For any two cubical $d$-polytopes $P_1$ and $P_2$, and facets $F_i$ of $P_i$, $i=1,2$, there exists a projective transformation $\phi$ and a $d$-tower $T$ of at most $4d$ cubes, such that $P:=P_1\cup T\cup \phi(P_2)$ is \emph{convex}, $P_1\cap T=F_1$ and $\phi(P_2)\cap T=\phi(F_2)$ are the top and bottom facets of $T$ respectively.
	\end{enumerate}
\end{theorem}

We apply this cubical connected sum operation to the cubical polytopes constructed recently in~\cite{AdinKN18}; the $f$-vectors of the latter approach the extremal rays of Adin's cone, which is conjectured to contain all $f$-vectors of $d$-polytopes~\cite{Adin96}.
The following density result for $f$-vectors of cubical polytopes then follows:
Let $\square^d$ denote the $d$-cube and $f(P)$ denote the $f$-vector of polytope $P$. Let $\mathcal{A}_d$ be the Adin cone (its apex is $f(\square^d)$ and its dimension is $\lfloor d/2\rfloor$ by the cubical Dehn-Sommerville relations~\cite{Adin96}).
\begin{theorem}[Ray density in Adin's cone]\label{thm:dense-f}
	For any $\epsilon>0$ and any $x\in \mathcal{A}_d$ there exists a cubical $d$-polytope $P$ such that the angle $\measuredangle xf(\square^d)f(P)$ is smaller than $\epsilon$.
\end{theorem}

Lastly, we note that our cubical connected sum construction endows the set of $f$-vectors of cubical $d$-polytopes with the structure of an affine semigroup (see~\cite{Zieg18}).

\paragraph{Outline} In Section~\ref{sec:prelim} we give preliminaries, we prove Theorem~\ref{thm:eq-cubes} in Section~\ref{sec:eq-cubes}, Theorem~\ref{thm:cubical.conn.sum} in Section~\ref{sec:cubicalCS} and Theorem~\ref{thm:dense-f} in Section~\ref{sec:Adin.cone.density}. We conclude with generalizations and related open questions in Section~\ref{sec:concluding}.

\section{Preliminaries}\label{sec:prelim}
For further background on polytopes see e.g.~\cite{Zieg95}.
\subsection{Two notions of equivalence of $d$-polytopes}

Let $P=\set{x\in\R^d}{Ax\leq b}$ be a $d$-polytope, with the origin in its {\bf interior} $P^{\circ}$. Denote by $r_1,\dots,r_m$ the rows of $A$. By scaling we may assume these are the {\bf facet outer normals}. The {\bf polar polytope}
$$P^{\triangle}=\set{y\in\R^d}{\langle y,x\rangle\leq 1\text{ for all }x\in P}=\conv(p_1,\dots,p_m)$$
is the $d$-polytope with vertices $p_1=\frac{1}{b_1}r_1,\dots,p_m=\frac{1}{b_m}r_m$.

A {\bf projective transformation} is a map
$$\varphi:\R^d\longrightarrow\R^d$$
defined by
$$x \mapsto \frac{Ax+b}{c^Tx+\alpha},$$
for some $A\in M_{d\times d}(\R)$, $b,c\in\R^d$, and $\alpha\in\R$ that satisfy
$$
\det\begin{pmatrix}
A & b\\
c^T & \alpha
\end{pmatrix}\neq 0.
$$
These transformations form a group under composition.
\begin{definition}
	Two $d$-polytopes $P$ and $Q$ are {\bf projectively equivalent} if there is a projective transformation $\varphi$ such that $Q=\varphi(P)$.
\end{definition}

We will need another notion of equivalence:
\begin{definition}
	Two $d$-polytopes $P$ and $Q$ are {\bf normally equivalent} if they have the same set of facet outer normals.
\end{definition}
In this case we say $Q=\psi(P)$ for a normal transformation $\psi$.
Thus, given a polytope $P$, any two polytopes normally equivalent to it differ by a normal transformation.
On the dual polytopes we say $Q^{\triangle}=\psi^{\triangle}(P^{\triangle})$ for a {\bf ray transformation}  $\psi^{\triangle}$ (it scales the vertices along the rays from the origin while preserving the combinatorial type).

\subsection{Connected sums of $d$-polytopes}
Suppose $P$ and $Q$ are $d$-polytopes whose intersection is a common facet $F=P\cap Q$ of both.
If $R=P\cup Q$ is convex then its proper faces are precisely the proper faces of either $P$ or $Q$, excluding $F$:
$$\mathrm{faces}(R)=\left( \mathrm{faces}(P)\cup \mathrm{faces}(Q)\right)\setminus \{F\}.$$

The following lemma, a proof of which can be found in~\cite[Lemma 3.2.4]{Rich96}, tells us when and how the connected sum of two polytopes can be formed.
\begin{lemma}\label{lem:connected_sum}
	Let $P$ and $Q$ be $d$-polytopes that have projectively equivalent facets $F_1$ and $F_2$ respectively. Then there exists a projective transformation $\varphi$ so that $P\cap \varphi(Q)=F_1=\varphi(F_2)$ and $R=P\cup \varphi(Q)$ is convex.
\end{lemma}
The combinatorial type of $R$ in Lemma~\ref{lem:connected_sum}
is called the {\bf connected sum} of $P$ and $Q$ along $F_1$ and $F_2$,
denoted $P\#_{F_1\sim F_2} Q$, or simply $P\#_F Q$ when the faces $F_1,F_2$ combinatorially isomorphic to $F$ are understood.

\subsection{Cubical polytopes}
We give just a brief reminder of the definitions of a cubical $d$-polytope and its $\hc$-vector and $\gc$-vector. For more details, in particular, for the construction used in section~\ref{sec:Adin.cone.density}, see~\cite{AdinKN18}.

A $d$-polytope $Q$ is {\bf cubical} if each of its proper faces is combinatorially a cube. Its {\bf $f$-polynomial} is defined by (note the shift of index!)
$$
f(Q,t)=\sum_{i=0}^{d-1} f_it^i
$$
where $f_i=f_i(Q)$ is the number of $i$-dimensional faces of $Q$.

We then define the {\bf short cubical $h$-polynomial}:
$$
\hsc(Q,t)=(1-t)^{d-1} f\left(Q,\frac{2t}{1-t}\right),
$$
and the {\bf cubical $h$-polynomial}
$$
\begin{aligned}
\hc(Q,t)=\sum_{i=0}^d \hc_it^i &=\frac{t(1-t)^{d-1}}{1+t} f\left(Q,\frac{2t}{1-t}\right) + 2^{d-1}\frac{1-(-t)^{d+1}}{1+t}.
\end{aligned}
$$

Adin~\cite{Adin96} has shown that $\hc(Q,t)$ is symmetric, that is
\begin{equation*}
\hc_i=\hc_{d-i}\quad (0\leq i\leq d).
\end{equation*}

These $\duhalf{d}$ equations are the {\bf cubical Dehn--Sommerville relations}.
We thus define the {\bf cubical $g$-vector} $\gc(Q)=(\gc_0,\dots,\gc_{\dhalf{d}})$ by
\begin{align*}
\gc_0=\hc_0=2^{d-1},& &  \gc_i&=\hc_i -\hc_{i-1}\quad\text{for } 1\leq i\leq\dhalf{d}.
\end{align*}
Adin conjectured
\begin{conjecture}[Question 2 in~\cite{Adin96}]
	For a cubical $d$-polytope we have
	\begin{equation}\label{eq:Adin.cone}
	\gc_i\geq 0\quad (1\leq i\leq \dhalf{d}).
	\end{equation}
\end{conjecture}

The cone~\eqref{eq:Adin.cone} is the nonnegative orthant in $\R^{\dhalf{d}}$, and its image under the map transforming $\gc$-vectors back into $f$-vectors yields the Adin cone $\mathcal{A}_d$ in $\R^d$.

In~\cite{AdinKN18}, for each $1\leq i\leq\dhalf{d}$, the authors exhibit a sequences of cubical $d$-polytopes whose corresponding sequence of $\gc$-vectors approaches the ray spanned by $e_i$. This translates into sequences of $f$-vectors approaching the extremal rays of $\mathcal{A}_d$.

\section{Any two combinatorial $d$-cubes are related by normal and projective transformations}\label{sec:eq-cubes}

We will use the following lemma, which describes the effect of a projective transformation on the polar polytope.
\begin{lemma}\label{lem:new0}
	Let $P$ be a $d$-polytope with $0\in P^{\circ}$, and $P^{\triangle}=\conv(p_1,\dots,p_m)$. Then for any $v\in P^{\circ}$ there exists a $d$-polytope $Q$ which is projectively equivalent to $P$, and $Q^{\triangle}=\conv(p_1+v,\dots,p_m+v)$.
\end{lemma}
\begin{proof}
	Consider the effect of a projective transformation $\varphi:\R^d\longrightarrow\R^d$ that takes $P$ to $Q$ (with $0\in Q^{\circ}$) on the polar polytopes $P^{\triangle}$ and $Q^{\triangle}$. It is easy check that the map
	$$\varphi^{\triangle}:\R^d\longrightarrow\R^d$$
	defined by
	$$x \mapsto \frac{A^Tx-c}{-b^Tx+\alpha}$$
	where $(\cdot)^T$ denotes the transpose, is a projective transformation that satisfies
	$$\varphi^{\triangle}(Q^{\triangle})=P^{\triangle}.$$
	Denote $\varphi^{-\triangle}=(\varphi^{\triangle})^{-1}$, so that
	$$Q^{\triangle}=\varphi^{-\triangle}(P^{\triangle}).$$
	Taking $A=I_{d\times d}$, $b=0$, $c=v$, and $\alpha=1$ produces a projective transformation $\varphi$ for which
	$$\varphi^{-\triangle}(x)=x+v$$
	and the claim follows.
\end{proof}

Let $Q=\set{x\in\R^d}{Ax\leq b}$ be a combinatorial $d$-cube, with the origin in its interior, and  $r_1,\dots,r_{2d}$ the rows of $A$, that is, the facet outer normals. We may assume that they are ordered by pairs of combinatorially opposite facets, that is, $r_i$ is normal to a facet opposite to the facet normal to $r_{i+1}$, for $i=1,3,\dots,2d-1$. The polar polytope $Q^{\triangle}$ is the combinatorial $d$-crosspolytope with vertices $p_1=\frac{1}{b_1}r_1,\dots,p_{2d}=\frac{1}{b_{2d}}r_{2d}$, and we denote by $l_i=[p_{2i-1},p_{2i}]$, for $1\leq i\leq d$, the line segments connecting the pairs of opposite vertices. The following proposition proves Theorem~\ref{thm:eq-cubes}$(a)$.

\begin{proposition}\label{prop:main}
	Let $Q$ and $Q'$ be two combinatorial $d$-cubes. Then there is a sequence $\phi_1,\dots\phi_s$ ($s\le 8d-1$) of projective and normal transformations such that
	$$Q'=(\phi_s\circ\dots\circ\phi_1)(Q).$$
\end{proposition}
\begin{proof}
	We present the sequence in terms of the polar $d$-crosspolytopes. For each pair of antipodal vertices of $P:=Q^{\triangle}$ we perform a sequence of $4$ transformations, alternating between projective and ray transformations,	arriving at a $d$-crosspolytope which is ray equivalent to the standard $d$-crosspolytope, namely the convex hull of the standard basis elements and their minuses. We refer to the sequence of $4$ transformations for the $i$-th pair of antipodal vertices as the {\em $i$-th iteration}. We denote the crosspolytope obtained after the $i$-th iteration by $P^{(i)}$, its vertices by $p_1^{(i)},p_2^{(i)},\dots ,p_{2d}^{(i)}$ and the line segments connecting its pairs of opposite vertices $p_{2j-1}^{(i)},p_{2j}^{(i)}$ by $l_j^{(i)}$ for $j=1,2,\dots d$.
	
	\begin{enumerate}
		\item Use Lemma~\ref{lem:new0} to translate the crosspolytope $P^{(i-1)}$ so that the origin lies on the interior of the line segment $l_i^{(i-1)}=[p_{2i-1}^{(i-1)},p_{2i}^{(i-1)}]$, say on its mid point to make a canonical choice; this projective transformation produces a polytope $P'$.
		
		\item For $P'=\{x|\ Ax\le b\}$ with vertex notation as in $P^{(i-1)}$, choose $c_{2i-1}\geq\frac{1}{b_{2i-1}}$ so that there exists an affine hyperplane $H_i$ orthogonal to $l_i$, which strictly separates $q_{2i-1}:=c_{2i-1}r_{2i-1}$ from $\vertices(P')\setminus\{p_{2i-1}\}$.
To make a canonical choice, let $c$ be the infimum of the possible values for such $c_{2i-1}$s, fix $c_{2i-1}=c+1$ and fix the $H_i$ as above that intersects the ray spanned by $q_{2i-1}$ at $(c+0.5)r_{2i-1}$.

Set $P''=\conv(q_{2i-1}\cup P')$. Then $P''$ is ray equivalent to $P'$.
		
		\item Again denote the vertices of $P''$ by $p_i$, in correspondence with the vertices of $P^{(i-1)}$, so $p_{2i-1}=q_{2i-1}$. Use again Lemma~\ref{lem:new0} to move the origin close enough to $p_{2i-1}$ along the segment $l_i$, that is, so that the origin and $p_{2i-1}$ are on the same side of the hyperplane $H_i$ of step (2). To make a canonical choice, move the origin to $(c+0.7)r_{2i-1}$.
Then
		$$
		H_i\cap P'' \cong \conv(H_i\cap \vspan(p_j)\mid j\in [2d]\setminus\{2i-1,2i\}).
		$$
		(Here $\cong$ means combinatorially equivalent). The resulted polytope $P'''$ is projectively equivalent to $P''$.
		
		\item Set $q_j:=H_i\cap \vspan(p_j)$ for $j\in [2d]\setminus\{2i-1,2i\}$. Then\\ $P^{(i)}=\conv(q_1,\ldots,q_{2i-2},p_{2i-1},p_{2i},q_{2i+1,\ldots,q_2d})$
		is ray equivalent to $P'''$.
	\end{enumerate}
	
	%
	%
	%
	
	The resulted crosspolytope $P^{(i)}$ has the property that all line segments, except $l^{(i)}_i$ lie on the hyperplane $H_i$, which is orthogonal to $l^{(i)}_i$. Furthermore, for all previous line segments $l^{(i)}_1,\dots,l^{(i)}_{i-1}$ the same property, achieved at the previous iterations, i.e., that all other line segments lie on the hyperplanes $H_1,\dots,H_{i-1}$ (respectively) still holds, because (i) these hyperplanes are spanned by rays, and the new points we choose at the $i$-th iteration are on the rays, and further (ii) the $1$st up to $(i-1)$th segments are just translated in the $i$th iteration.
	
	After performing this process for every pair of antipodal vertices we obtain a combinatorial $d$-crosspolytope, with segments $l_1,\dots,l_d$, such that, for each $1\leq i\leq d$, there exists an affine hyperplane $H_i$, which is orthogonal to $l_i$, and contains all other segments $l_j$, $j\neq i$. It follows that the segments $l_1,\dots,l_d$ all intersect in a point, and are pairwise orthogonal.
	
	To see that the line segments are pairwise orthogonal, note that if the line segment $l_i$ was orthogonal to the line segment $l_j$ before performing the $i$-th iteration, then the new line segment $l_i^{(i)}$ is orthogonal to the new line segment $l_j^{(i)}$.
	
	To see that all line segments intersect in a point, consider the affine space spanned by the line segments, constructed sequentially. We start with some line segment (it spans an affine space of dimension $1$), then add a second line segment, which can raise the dimension by $0$, $1$ or $2$, and so on. Note that at each step the dimension cannot grow by $0$ (because each line segment is orthogonal to the space spanned by all other segments), and since we have $d$ line segments in $\R^d$ the total dimension is at most $d$, so at each step the dimension cannot grow by $2$ either, thus, at each step, the dimension grows by $1$. Note that this argument is valid for any ordering of the line segments, so any two of the segments intersect in a point. Using their pairwise orthogonality, this is the same point for all pairs of segments.
	
	We perform the same procedure for $Q'^{\triangle}$ to get a combinatorial $d$-crosspolytope which is normally equivalent to the standard $d$-crosspolytope. To finish, we do a final normal transformation to concatenate the two sequences of transformations performed on $Q$ and on $Q'$. In fact, the resulted 3 normal transformations in a row can be replaced by a single one. This algorithm gives $s=8d-1$.
\end{proof}


In the following figure we give an illustration of a single iteration for an octahedron. The line segment $l=vv'$ is colored black. The red lines represent the rays from the origin (the red point) on which the vertices lie.\\
	{\bf Perform ${\phi_1}$}: the origin now lies on $l$.\\
	{\bf Perform $\phi_2$}: the vertex $v$ is moved along its ray so that a hyperplane as in the next step exists.\\
	{\bf Perform $\phi_3$}: the hyperplane $H_l$ is orthogonal to $l$, and seperates $v$ and the origin, from the other vertices.\\
	{\bf Perform $\phi_4$}: the resulted octahedron has the property that all diagonals, besides $l$, lie on a hyperplane, which is orthogonal to $l$.

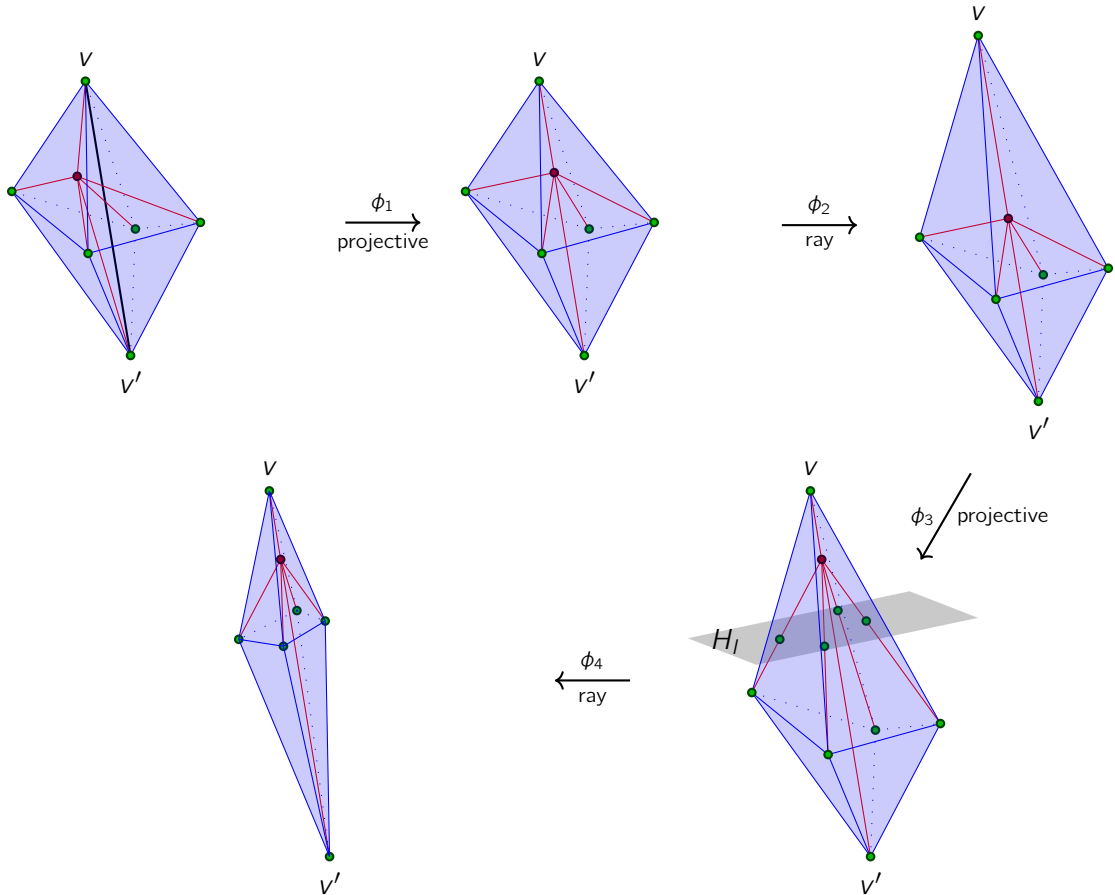
\begin{center}
\begin{minipage}{0.25\linewidth}
	\begin{tikzpicture}%
	[x={(-0.291565cm, 0.113036cm)},
	y={(-0.943073cm, -0.200076cm)},
	z={(0.160011cm, -0.973238cm)},
	scale=1.25,
	back/.style={loosely dotted, thin},
	diags/.style={color=black!95!black, thick},
	edge/.style={color=blue!95!black, thin},
	ray/.style={color=red!95!black, thin},
	facet/.style={fill=blue!95!black,fill opacity=0.2},
	vertex/.style={inner sep=1pt,circle,draw=green!25!black,fill=green!75!black,thick,anchor=base},
	vorigin/.style={inner sep=1pt,circle,draw=red!25!black,fill=red!75!black,thick,anchor=base}]
	%
	%
	\coordinate (a) at (1.00000, 0.00000, 1.00000);
	\coordinate (b) at (-1.00000, 0.00000, 0.50000);
	\coordinate (c) at (0.00000, 1.00000, 0.00000);
	\coordinate (d) at (0.00000, -1.00000, 0.75000);
	\coordinate (e) at (0.00000, 0.00000, 2.00000);
	\coordinate (f) at (0.00000, 0.00000, -1.00000);
	\coordinate (o) at (-0.5, 0.4, -0.1);
	\draw[edge,diags] (0.00000, 0.00000, 2.00000) -- (0.00000, 0.00000, -1.00000);
	\draw[ray] (o) -- (a);
	\draw[ray] (o) -- (b);
	\draw[ray] (o) -- (c);
	\draw[ray] (o) -- (d);
	\draw[ray] (o) -- (e);
	\draw[ray] (o) -- (f);
	\node[vorigin] at (o)     {};
	\draw[edge,back] (-1.00000, 0.00000, 0.50000) -- (0.00000, 1.00000, 0.00000);
	\draw[edge,back] (-1.00000, 0.00000, 0.50000) -- (0.00000, -1.00000, 0.75000);
	\draw[edge,back] (-1.00000, 0.00000, 0.50000) -- (0.00000, 0.00000, 2.00000);
	\draw[edge,back] (-1.00000, 0.00000, 0.50000) -- (0.00000, 0.00000, -1.00000);
	\node[vertex] at (-1.00000, 0.00000, 0.50000)     {};
	\fill[facet] (0.00000, 0.00000, 2.00000) -- (1.00000, 0.00000, 1.00000) -- (0.00000, -1.00000, 0.75000) -- cycle {};
	\fill[facet] (0.00000, 0.00000, -1.00000) -- (1.00000, 0.00000, 1.00000) -- (0.00000, -1.00000, 0.75000) -- cycle {};
	\fill[facet] (0.00000, 0.00000, 2.00000) -- (1.00000, 0.00000, 1.00000) -- (0.00000, 1.00000, 0.00000) -- cycle {};
	\fill[facet] (0.00000, 0.00000, -1.00000) -- (1.00000, 0.00000, 1.00000) -- (0.00000, 1.00000, 0.00000) -- cycle {};
	\draw[edge] (1.00000, 0.00000, 1.00000) -- (0.00000, 1.00000, 0.00000);
	\draw[edge] (1.00000, 0.00000, 1.00000) -- (0.00000, -1.00000, 0.75000);
	\draw[edge] (1.00000, 0.00000, 1.00000) -- (0.00000, 0.00000, 2.00000);
	\draw[edge] (1.00000, 0.00000, 1.00000) -- (0.00000, 0.00000, -1.00000);
	\draw[edge] (0.00000, 1.00000, 0.00000) -- (0.00000, 0.00000, 2.00000);
	\draw[edge] (0.00000, 1.00000, 0.00000) -- (0.00000, 0.00000, -1.00000);
	\draw[edge] (0.00000, -1.00000, 0.75000) -- (0.00000, 0.00000, 2.00000);
	\draw[edge] (0.00000, -1.00000, 0.75000) -- (0.00000, 0.00000, -1.00000);
	\node[vertex] at (1.00000, 0.00000, 1.00000)     {};
	\node[vertex] at (0.00000, 1.00000, 0.00000)     {};
	\node[vertex] at (0.00000, -1.00000, 0.75000)     {};
	\node[vertex,label=below:{$v'$}] at (0.00000, 0.00000, 2.00000)     {};
	\node[vertex,label=above:{$v$}] at (0.00000, 0.00000, -1.00000)     {};
	\end{tikzpicture}
\end{minipage}
\begin{minipage}{0.1\linewidth}
	\begin{tikzpicture}
	\draw[->,thick] (0,0,0) -- (1,0,0) node[draw=none,fill=none,font=\scriptsize,midway,below] {projective} node[draw=none,fill=none,font=\scriptsize,midway,above] {$\phi_1$};
	\end{tikzpicture}
\end{minipage}
\begin{minipage}{0.25\linewidth}
	\begin{tikzpicture}%
		[x={(-0.291565cm, 0.113036cm)},
		y={(-0.943073cm, -0.200076cm)},
		z={(0.160011cm, -0.973238cm)},
		scale=1.25,
		back/.style={loosely dotted, thin},
		diags/.style={color=black!95!black, thick},
		edge/.style={color=blue!95!black, thin},
		ray/.style={color=red!95!black, thin},
		facet/.style={fill=blue!95!black,fill opacity=0.2},
		vertex/.style={inner sep=1pt,circle,draw=green!25!black,fill=green!75!black,thick,anchor=base},
		vorigin/.style={inner sep=1pt,circle,draw=red!25!black,fill=red!75!black,thick,anchor=base}]
		%
		\coordinate (a) at (1.00000, 0.00000, 1.00000);
		\coordinate (b) at (-1.00000, 0.00000, 0.50000);
		\coordinate (c) at (0.00000, 1.00000, 0.00000);
		\coordinate (d) at (0.00000, -1.00000, 0.75000);
		\coordinate (e) at (0.00000, 0.00000, 2.00000);
		\coordinate (f) at (0.00000, 0.00000, -1.00000);
		\coordinate (o) at (0, 0, 0);
		\draw[ray] (o) -- (a);
		\draw[ray] (o) -- (b);
		\draw[ray] (o) -- (c);
		\draw[ray] (o) -- (d);
		\draw[ray] (o) -- (e);
		\draw[ray] (o) -- (f);
		\node[vorigin] at (o)     {};
		\draw[edge,back] (-1.00000, 0.00000, 0.50000) -- (0.00000, 1.00000, 0.00000);
		\draw[edge,back] (-1.00000, 0.00000, 0.50000) -- (0.00000, -1.00000, 0.75000);
		\draw[edge,back] (-1.00000, 0.00000, 0.50000) -- (0.00000, 0.00000, 2.00000);
		\draw[edge,back] (-1.00000, 0.00000, 0.50000) -- (0.00000, 0.00000, -1.00000);
		\node[vertex] at (-1.00000, 0.00000, 0.50000)     {};
		\fill[facet] (0.00000, 0.00000, 2.00000) -- (1.00000, 0.00000, 1.00000) -- (0.00000, -1.00000, 0.75000) -- cycle {};
		\fill[facet] (0.00000, 0.00000, -1.00000) -- (1.00000, 0.00000, 1.00000) -- (0.00000, -1.00000, 0.75000) -- cycle {};
		\fill[facet] (0.00000, 0.00000, 2.00000) -- (1.00000, 0.00000, 1.00000) -- (0.00000, 1.00000, 0.00000) -- cycle {};
		\fill[facet] (0.00000, 0.00000, -1.00000) -- (1.00000, 0.00000, 1.00000) -- (0.00000, 1.00000, 0.00000) -- cycle {};
		\draw[edge] (1.00000, 0.00000, 1.00000) -- (0.00000, 1.00000, 0.00000);
		\draw[edge] (1.00000, 0.00000, 1.00000) -- (0.00000, -1.00000, 0.75000);
		\draw[edge] (1.00000, 0.00000, 1.00000) -- (0.00000, 0.00000, 2.00000);
		\draw[edge] (1.00000, 0.00000, 1.00000) -- (0.00000, 0.00000, -1.00000);
		\draw[edge] (0.00000, 1.00000, 0.00000) -- (0.00000, 0.00000, 2.00000);
		\draw[edge] (0.00000, 1.00000, 0.00000) -- (0.00000, 0.00000, -1.00000);
		\draw[edge] (0.00000, -1.00000, 0.75000) -- (0.00000, 0.00000, 2.00000);
		\draw[edge] (0.00000, -1.00000, 0.75000) -- (0.00000, 0.00000, -1.00000);
		\node[vertex] at (1.00000, 0.00000, 1.00000)     {};
		\node[vertex] at (0.00000, 1.00000, 0.00000)     {};
		\node[vertex] at (0.00000, -1.00000, 0.75000)     {};
		\node[vertex,label=below:{$v'$}] at (0.00000, 0.00000, 2.00000)     {};
		\node[vertex,label=above:{$v$}] at (0.00000, 0.00000, -1.00000)     {};
	\end{tikzpicture}
\end{minipage}
\begin{minipage}{0.1\linewidth}
	\begin{tikzpicture}
	\draw[->,thick] (0,0,0) -- (1,0,0) node[draw=none,fill=none,font=\scriptsize,midway,below] {ray} node[draw=none,fill=none,font=\scriptsize,midway,above] {$\phi_2$};
	\end{tikzpicture}
\end{minipage}
\begin{minipage}{0.25\linewidth}
	\begin{tikzpicture}%
	[x={(-0.291565cm, 0.113036cm)},
	y={(-0.943073cm, -0.200076cm)},
	z={(0.160011cm, -0.973238cm)},
	scale=1.25,
	back/.style={loosely dotted, thin},
	diags/.style={color=black!95!black, thick},
	edge/.style={color=blue!95!black, thin},
	ray/.style={color=red!95!black, thin},
	facet/.style={fill=blue!95!black,fill opacity=0.2},
	vertex/.style={inner sep=1pt,circle,draw=green!25!black,fill=green!75!black,thick,anchor=base},
	vorigin/.style={inner sep=1pt,circle,draw=red!25!black,fill=red!75!black,thick,anchor=base}]
	%
	%
	\coordinate (a) at (1.00000, 0.00000, 1.00000);
	\coordinate (b) at (-1.00000, 0.00000, 0.50000);
	\coordinate (c) at (0.00000, 1.00000, 0.00000);
	\coordinate (d) at (0.00000, -1.00000, 0.75000);
	\coordinate (e) at (0.00000, 0.00000, 2.00000);
	\coordinate (f) at (0.00000, 0.00000, -2.00000);
	\coordinate (o) at (0, 0, 0);
	\draw[ray] (o) -- (a);
	\draw[ray] (o) -- (b);
	\draw[ray] (o) -- (c);
	\draw[ray] (o) -- (d);
	\draw[ray] (o) -- (e);
	\draw[ray] (o) -- (f);
	\node[vorigin] at (o)     {};
	\draw[edge,back] (-1.00000, 0.00000, 0.50000) -- (0.00000, 1.00000, 0.00000);
	\draw[edge,back] (-1.00000, 0.00000, 0.50000) -- (0.00000, -1.00000, 0.75000);
	\draw[edge,back] (-1.00000, 0.00000, 0.50000) -- (0.00000, 0.00000, 2.00000);
	\draw[edge,back] (-1.00000, 0.00000, 0.50000) -- (0.00000, 0.00000, -2.00000);
	\node[vertex] at (-1.00000, 0.00000, 0.50000)     {};
	\fill[facet] (0.00000, 0.00000, 2.00000) -- (1.00000, 0.00000, 1.00000) -- (0.00000, -1.00000, 0.75000) -- cycle {};
	\fill[facet] (0.00000, 0.00000, -2.00000) -- (1.00000, 0.00000, 1.00000) -- (0.00000, -1.00000, 0.75000) -- cycle {};
	\fill[facet] (0.00000, 0.00000, 2.00000) -- (1.00000, 0.00000, 1.00000) -- (0.00000, 1.00000, 0.00000) -- cycle {};
	\fill[facet] (0.00000, 0.00000, -2.00000) -- (1.00000, 0.00000, 1.00000) -- (0.00000, 1.00000, 0.00000) -- cycle {};
	\draw[edge] (1.00000, 0.00000, 1.00000) -- (0.00000, 1.00000, 0.00000);
	\draw[edge] (1.00000, 0.00000, 1.00000) -- (0.00000, -1.00000, 0.75000);
	\draw[edge] (1.00000, 0.00000, 1.00000) -- (0.00000, 0.00000, 2.00000);
	\draw[edge] (1.00000, 0.00000, 1.00000) -- (0.00000, 0.00000, -2.00000);
	\draw[edge] (0.00000, 1.00000, 0.00000) -- (0.00000, 0.00000, 2.00000);
	\draw[edge] (0.00000, 1.00000, 0.00000) -- (0.00000, 0.00000, -2.00000);
	\draw[edge] (0.00000, -1.00000, 0.75000) -- (0.00000, 0.00000, 2.00000);
	\draw[edge] (0.00000, -1.00000, 0.75000) -- (0.00000, 0.00000, -2.00000);
	\node[vertex] at (1.00000, 0.00000, 1.00000)     {};
	\node[vertex] at (0.00000, 1.00000, 0.00000)     {};
	\node[vertex] at (0.00000, -1.00000, 0.75000)     {};
	\node[vertex,label=below:{$v'$}] at (0.00000, 0.00000, 2.00000)     {};
	\node[vertex,label=above:{$v$}] at (0.00000, 0.00000, -2.00000)     {};
	\end{tikzpicture}
\end{minipage}

\begin{minipage}{0.25\linewidth}
	\begin{tikzpicture}%
	[x={(-0.291565cm, 0.113036cm)},
	y={(-0.943073cm, -0.200076cm)},
	z={(0.160011cm, -0.973238cm)},
	scale=1.25,
	back/.style={loosely dotted, thin},
	diags/.style={color=black!95!black, thick},
	edge/.style={color=blue!95!black, thin},
	ray/.style={color=red!95!black, thin},
	facet/.style={fill=blue!95!black,fill opacity=0.2},
	vertex/.style={inner sep=1pt,circle,draw=green!25!black,fill=green!75!black,thick,anchor=base},
	vorigin/.style={inner sep=1pt,circle,draw=red!25!black,fill=red!75!black,thick,anchor=base}]
	\coordinate (a) at (1, -0.365, -0.5);
	\coordinate (b) at (-1, 0.405 , -0.5);
	\coordinate (c) at (0.00000, 0.6, -0.5);
	\coordinate (d) at (0.00000, -0.375, -0.5);
	\coordinate (e) at (0.00000, 0.00000, 2.00000);
	\coordinate (f) at (0.00000, 0.00000, -2.00000);
	\coordinate (o) at (0, 0, -1.25);
	\draw[ray] (o) -- (a);
	\draw[ray] (o) -- (b);
	\draw[ray] (o) -- (c);
	\draw[ray] (o) -- (d);
	\draw[ray] (o) -- (e);
	\draw[ray] (o) -- (f);
	\node[vorigin] at (o)     {};
	\node[vertex] at (1, -0.365, -0.5)     {};
	\node[vertex] at (-1, 0.405 , -0.5)     {};
	\node[vertex] at (0.00000, 0.6, -0.5)     {};
	\node[vertex] at (0.00000, -0.375, -0.5)     {};
	\node[vertex,label=below:{$v'$}] at (0.00000, 0.00000, 2.00000)     {};
	\node[vertex,label=above:{$v$}] at (0.00000, 0.00000, -2.00000)     {};
	\draw[edge,back] (a) -- (c);
	\draw[edge,back] (a) -- (d);
	\draw[edge,back] (a) -- (e);
	\draw[edge,back] (a) -- (f);
	\fill[facet] (e) -- (b) -- (c) -- cycle {};
	\fill[facet] (f) -- (b) -- (c) -- cycle {};
	\fill[facet] (e) -- (b) -- (d) -- cycle {};
	\fill[facet] (f) -- (b) -- (d) -- cycle {};
	\draw[edge] (b) -- (c);
	\draw[edge] (b) -- (d);
	\draw[edge] (b) -- (e);
	\draw[edge] (b) -- (f);
	\draw[edge] (c) -- (e);
	\draw[edge] (c) -- (f);
	\draw[edge] (d) -- (e);
	\draw[edge] (d) -- (f);
	\end{tikzpicture}
\end{minipage}
\begin{minipage}{0.1\linewidth}
	\begin{tikzpicture}
	\draw[->,thick] (1,0,0) -- (0,0,0) node[draw=none,fill=none,font=\scriptsize,midway,below] {ray} node[draw=none,fill=none,font=\scriptsize,midway,above] {$\phi_4$};
	\end{tikzpicture}
\end{minipage}
\begin{minipage}{0.25\linewidth}
	\begin{tikzpicture}%
	[x={(-0.291565cm, 0.113036cm)},
	y={(-0.943073cm, -0.200076cm)},
	z={(0.160011cm, -0.973238cm)},
	scale=1.25,
	back/.style={loosely dotted, thin},
	diags/.style={color=black!95!black, thick},
	edge/.style={color=blue!95!black, thin},
	ray/.style={color=red!95!black, thin},
	facet/.style={fill=blue!95!black,fill opacity=0.2},
	vertex/.style={inner sep=1pt,circle,draw=green!25!black,fill=green!75!black,thick,anchor=base},
	vorigin/.style={inner sep=1pt,circle,draw=red!25!black,fill=red!75!black,thick,anchor=base}]
	\coordinate (a) at (1.00000, 0.00000, 1.00000);
	\coordinate (b) at (-1.00000, 0.00000, 0.50000);
	\coordinate (c) at (0.00000, 1.00000, 0.00000);
	\coordinate (d) at (0.00000, -1.00000, 0.75000);
	\coordinate (e) at (0.00000, 0.00000, 2.00000);
	\coordinate (f) at (0.00000, 0.00000, -2.00000);
	\coordinate (o) at (0, 0, -1.25);
	\draw[ray] (o) -- (a);
	\draw[ray] (o) -- (b);
	\draw[ray] (o) -- (c);
	\draw[ray] (o) -- (d);
	\draw[ray] (o) -- (e);
	\draw[ray] (o) -- (f);
	\node[vorigin] at (o)     {};
	\fill[fill=gray!95!black,fill opacity=0.400000] (1.250000, 1.250000, -0.50000) -- (1.250000, -1.250000, -0.500000) -- (-1.250000, -1.250000, -0.5) -- (-1.250000, 1.250000, -0.5) -- cycle {};
	\node[fill=none] at (0,1.2,-0.6) {$H_l$};
	\node[vertex] at (0.00000, 0.6, -0.5)     {};
	\node[vertex] at (0.00000, -0.375, -0.5)     {};
	\node[vertex] at (1, -0.365, -0.5)     {};
	\node[vertex] at (-1, 0.405 , -0.5)     {};
	\draw[edge,back] (-1.00000, 0.00000, 0.50000) -- (0.00000, 1.00000, 0.00000);
	\draw[edge,back] (-1.00000, 0.00000, 0.50000) -- (0.00000, -1.00000, 0.75000);
	\draw[edge,back] (-1.00000, 0.00000, 0.50000) -- (0.00000, 0.00000, 2.00000);
	\draw[edge,back] (-1.00000, 0.00000, 0.50000) -- (0.00000, 0.00000, -2.00000);
	\node[vertex] at (-1.00000, 0.00000, 0.50000)     {};
	\fill[facet] (0.00000, 0.00000, 2.00000) -- (1.00000, 0.00000, 1.00000) -- (0.00000, -1.00000, 0.75000) -- cycle {};
	\fill[facet] (0.00000, 0.00000, -2.00000) -- (1.00000, 0.00000, 1.00000) -- (0.00000, -1.00000, 0.75000) -- cycle {};
	\fill[facet] (0.00000, 0.00000, 2.00000) -- (1.00000, 0.00000, 1.00000) -- (0.00000, 1.00000, 0.00000) -- cycle {};
	\fill[facet] (0.00000, 0.00000, -2.00000) -- (1.00000, 0.00000, 1.00000) -- (0.00000, 1.00000, 0.00000) -- cycle {};
	\draw[edge] (1.00000, 0.00000, 1.00000) -- (0.00000, 1.00000, 0.00000);
	\draw[edge] (1.00000, 0.00000, 1.00000) -- (0.00000, -1.00000, 0.75000);
	\draw[edge] (1.00000, 0.00000, 1.00000) -- (0.00000, 0.00000, 2.00000);
	\draw[edge] (1.00000, 0.00000, 1.00000) -- (0.00000, 0.00000, -2.00000);
	\draw[edge] (0.00000, 1.00000, 0.00000) -- (0.00000, 0.00000, 2.00000);
	\draw[edge] (0.00000, 1.00000, 0.00000) -- (0.00000, 0.00000, -2.00000);
	\draw[edge] (0.00000, -1.00000, 0.75000) -- (0.00000, 0.00000, 2.00000);
	\draw[edge] (0.00000, -1.00000, 0.75000) -- (0.00000, 0.00000, -2.00000);
	\node[vertex] at (1.00000, 0.00000, 1.00000)     {};
	\node[vertex] at (0.00000, 1.00000, 0.00000)     {};
	\node[vertex] at (0.00000, -1.00000, 0.75000)     {};
	\node[vertex,label=below:{$v'$}] at (0.00000, 0.00000, 2.00000)     {};
	\node[vertex,label=above:{$v$}] at (0.00000, 0.00000, -2.00000)     {};
	\draw[->,thick] (-1,-1.5,-2) -- (-1,-0.8,-1.2)
	node[draw=none,fill=none,font=\scriptsize,midway,right] {projective}
	node[draw=none,fill=none,font=\scriptsize,midway,left] {$\phi_3$};
	\end{tikzpicture}
\end{minipage}
\end{center}

We now deduce Theorem~\ref{thm:eq-cubes}$(b)$, that is, that the realization space $R$ is contractible. Consider the point $p$ in $R$ corresponding to the standard cube. Consider the sequence of $4d+2$ transformations that take the point $x\in R$ to the point $p$: after performing the duals of the first $4d$ transformation in above proof, for the last two transformations, the first is normal that changes from a box to a unit cube, and the second is projective, in fact an isometry to the standard cube (axis parallel, unit, with the origin at its center of mass).
This sequence can be seen as a continuous path in $R$ from $x$ to $p$, where we take the dual of each translation and each ray scaling done linearly in one unit time; likewise for the last two transformations. The resulted map $f:R\times[0,4d+2]\longrightarrow R$ is a homotopy from $R$ to the point $p$: it is indeed continuous by the canonical choices in steps (1--4) for each iteration of the algorithm in the proof of Proposition~\ref{prop:main}. (If $y$ is a cube nearby $x$, then for any vertex in $x$ its unique nearby vertex in $y$ gets the same combinatorial labeling. Given the algorithm for $x$, this determines the algorithm for $y$, and hence the path from $y$ to $p$ in $R$.)
$\ \blacksquare$

\section{A cubical connector $d$-polytope and the $C$-connected sum}\label{sec:cubicalCS}

\begin{definition}
	A {\bf $d$-tower} of $s$ cubes is a cubical stacked $d$-polytope $T$ obtained by stacking on the facet opposite to the facet stacked on in the previous step.
	
	Explicitly, for $s=1$ it is just a $d$-cube. Mark some two opposite facets as bottom and top. For $s>1$, a $d$-tower of $s$ cubes is obtained from a $d$-tower of $s-1$ cubes with bottom facet $F$ and top facet $F'$ by stacking a $d$-cube onto $F'$.	
Then the polytope $T$ has a unique {\bf bottom} facet and a unique {\bf top} facet.
\end{definition}

Given two combinatorial $(d-1)$-cubes $Q_1$ and $Q_2$, we use Proposition~\ref{prop:main} to construct a $d$-tower having bottom facet $Q'_1$ and top facet $Q'_2$, with $Q'_i$ projectively equivalent to $Q_i$, $i=1,2$ . The following lemma shows how to translate each normal transformation from Proposition~\ref{prop:main} into a $d$-tower of $1$ cube.

\begin{lemma}\label{lem:comb_cubes}
	Let $Q_1$ and $Q_2$ be two combinatorial $(d-1)$-cubes which are normally equivalent.
	Then there exists a $d$-cube $Q$ in which $Q_1$ 
	and $Q_2$ (both realized in $\R^d$) are opposite facets.
\end{lemma}
\begin{proof}
	
	Assume that both $Q_1$ and $Q_2$ are realized in $\R^d$ on the last coordinate $=0$ hyperplane. Lift the vertices of $Q_2$ (say to height $1$), and take the convex hull, denote it by $Q$.
	
	Here is an explicit description of $Q$:	
	Let $A_1,A_2\in \R^{(2d-2)\times (d-1)}$ and $b_1,b_2\in\R^{2d-2}$ be such that
	\begin{equation*}
	Q_1=\set{x\in\R^{d-1}}{A_1x\leq b_1},\qquad
	Q_2=\set{x\in\R^{d-1}}{A_2x\leq b_2}.
	\end{equation*}
	The fact that $Q_1$ and $Q_2$ are normally equivalent means that $A_1=A_2$.
	We define
	\begin{equation}
	A=\left(
	\begin{array}{ccc|c}
	& & & \mid\\
	& A_1 & & b_1-b_2\\
	& & & \mid\\
	\hline
	0 & \cdots & 0 & 1\\
	0 & \cdots & 0 & -1
	\end{array}\right),\qquad
	b=\left(
	\begin{array}{c}
	\mid\\
	b_1\\
	\mid\\
	\hline
	1\\
	0
	\end{array}\right),
	\end{equation}
	and
	$$
	Q=\set{x\in\R^d}{Ax\leq b}.\qedhere
	$$
\end{proof}

Applying Lemma~\ref{lem:comb_cubes} to each of the normal transformations in Proposition~\ref{prop:main}, and Lemma~\ref{lem:connected_sum}
to glue each such new $d$-cube to the previously constructed polytope so that the result is again a convex polytope,
we conclude Theorem~\ref{thm:cubical.conn.sum}
(here $s$ is as in Proposition~\ref{prop:main}):
\begin{corollary}
	Let $Q$ and $Q'$ be two combinatorial $(d-1)$-cubes. Then there is a $d$-tower of
$4d$ cubes with bottom facet projectively equivalent to $Q$ and top facet projectively equivalent to $Q'$. We call this tower a {\bf cubical connector} and denote it $C(Q,Q')$.
\end{corollary}

\begin{definition}
	Let $Q_1$ and $Q_2$ be cubical $d$-polytopes. Let $F_1$ be a facet of $Q_1$, $F_2$ a facet of $Q_2$, and $C=C(F_1,F_2)$ the appropriate cubical connector (a tower of $4d$ cubes). The {\bf $\mathbf C$-connected sum} $Q=Q_1\# Q_2$ is the cubical $d$-polytope obtained by taking the connected sum $Q_1\#_{F_1}C\#_{F_2}Q_2$.
\end{definition}


\section{Filling the $\gc$-cone
}\label{sec:Adin.cone.density}

We apply the connected sum construction to appropriate AKN-polytopes (see~\cite{AdinKN18}) thus obtaining sequences of cubical $d$-polytopes with corresponding $\gc$-vector sequences approaching any ray in the nonnegative orthant of $\R^{\dhalf{d}}$.

\begin{lemma}\label{lem:gc_of_cs}
	Let $Q_1\# Q_2$ be a $C$-connected sum then
	\begin{equation}\label{eq:gc_of_cs}
	\begin{aligned}
	\gc_1(Q_1\# Q_2)&=\gc_1(Q_1) + \gc_1(Q_2) + (4d+1)2^{d-1},\\
	\gc_i(Q_1\# Q_2)&=\gc_i(Q_1) + \gc_i(Q_2)\qquad (2\leq i\leq \dhalf{d}).
	\end{aligned}
	\end{equation}
\end{lemma}
\begin{proof}
	Let us first observe that for the (usual) connected sum $Q\#_F Q'$, when $Q$ and $Q'$ are cubical $d$-polytopes we have
	\begin{equation*}
	f(Q\#_F Q',t)=f(Q,t)+f(Q',t)-f(\square^{d-1},t)-t^{d-1},
	\end{equation*}
	and
	\begin{equation*}
	\hsc(Q\#_F Q',t)=(1-t)^df\left(Q\#_F Q',\frac{2t}{1-t}\right).
	\end{equation*}
	
	It follows that
	\begin{equation*}
	\hsc(Q\#_F Q',t)=\hsc(Q,t)+\hsc(Q',t)+(t-1)\left(\hsc(\square^{d-1},t)+2^{d-1}t^{d-1}\right),
	\end{equation*}
	and so, for $1\leq i\leq d-2$, we have
	\begin{align*}
	\hsc_i(Q\#_F Q')&=\hsc_i(Q)+\hsc_i(Q')+\hsc_{i-1}(\square^{d-1})-\hsc_i(\square^{d-1})\\
	&=\hsc_i(Q)+\hsc_i(Q').
	\end{align*}
	
	It immediately follows that
	\begin{align*}
	\hc_i(Q\#_F Q')&=\hc_i(Q)+\hc_i(Q')\qquad (1\leq i\leq d-1),\\
	\gc_i(Q\#_F Q')&=\gc_i(Q)+\gc_i(Q')\qquad (2\leq i\leq \dhalf{d}).
	\end{align*}
	
	We can now analyze the $C$-connected sum $Q_1\# Q_2$. Since
	\begin{equation*}
	Q_1\# Q_2=Q_1\#_{F_1} C\#_{F_2} Q_2,
	\end{equation*}
	we have
	\begin{equation*}
	\gc_i(Q_1\# Q_2)=\gc_i(Q_1) + \gc_i(C) + \gc_i(Q_2)\qquad (2\leq i\leq \dhalf{d}),
	\end{equation*}
	and since $C$ is cubical stacked we have 
$\gc_i(C)=0$ for $2\leq i\leq \dhalf{d}$, so we obtain~\eqref{eq:gc_of_cs} as required for $2\leq i\leq \dhalf{d}$. For $i=1$ one computes directly, using $\gc_1=f_0-2^d$.
\end{proof}

The following proves Theorem~\ref{thm:dense-f}:
\begin{proposition}
	Let $r$ be any ray in the nonnegative orthant in $\mathbb{R}^{\dhalf{d}}$. Then there exists a sequence $\{Q_n\}_{n=1}^{\infty}$ of cubical $d$-polytopes with the sequence $\{\gc(Q_n)\}_{n=1}^{\infty}$ approaching $r$.
\end{proposition}
\begin{proof}
	To construct the sequence $Q_n$ approaching $r$, the ray spanned by $(s_1,\dots ,s_{\dhalf{d}})$, we start by constructing a sequence having the correct ratio between the $\dhalf{d}$-th coordinate and the $(\dhalf{d}-1)$-th coordinate:	
	Take
	$$
	Q_m=Q(\dhalf{d},d,m),\quad m\to\infty\qquad\text{and}\qquad Q'_l=Q(\dhalf{d}-1,d,l),\quad l\to\infty,
	$$
	and recall from~\cite{AdinKN18} that
	\begin{align*}
	\gc_{\dhalf{d}}(Q_m)=2^mm^{\dhalf{d}-1}+o(2^mm^{\dhalf{d}-1}),&\qquad \gc_{\dhalf{d}-1}(Q_m)=o(2^mm^{\dhalf{d}-1})\\
	\gc_{\dhalf{d}}(Q'_l)=0,&\qquad
	\gc_{\dhalf{d}-1}(Q'_l)=2^ll^{\dhalf{d}-2}+o(2^ll^{\dhalf{d}-2}).
	\end{align*}
	
	Let $c=\frac{s_{\dhalf{d}-1}}{s_{\dhalf{d}}}$. For each $m\geq d$, let $l=\lceil \log c + m+\log m\rceil$, and take the corresponding subsequence of $Q'_l$'s (we abuse notation and denote it again by $Q'_l$). We have
	\begin{align*}
	\gc_{\dhalf{d}-1}(Q'_l)&=2^ll^{\dhalf{d}-2}+o(2^ll^{\dhalf{d}-2})\\
	&=2^{\lceil \log c + m+\log m\rceil}(\lceil m+\log m+\log c\rceil)^{\dhalf{d}-2}  + o(2^ll^{\dhalf{d}-2})\\
	&=c2^mm^{\dhalf{d}-1}+o(2^mm^{\dhalf{d}-1})
	\end{align*}
	and so taking
	$$Q_n=Q_m\#Q'_l,\quad n\to\infty,$$ using Lemma~\ref{lem:gc_of_cs} we obtain
	\begin{equation*}
	\lim_{n \to \infty} \frac{\gc_{\dhalf{d}-1}(Q_n)}{\gc_{\dhalf{d}}(Q_n)} = \frac{s_{\dhalf{d}-1}}{s_{\dhalf{d}}}.
	\end{equation*}
	
	Do the same with the new sequence and an AKN-sequence approaching the $(\dhalf{d}-2)$-th ray, etc. Note that proceeding in this way (from the last coordinate backwards) does not influence the ratios already taken care of, because the $\gc$-entries are $0$ after the dominating coordinate in the AKN construction.
	
	Since the formula for $\gc_1$ in Lemma~\ref{lem:gc_of_cs} is different we use $c=\frac{s_1}{s_2}+2^{d-m}\left(1-\frac{1}{m}\right)$ in the last step.
\end{proof}

\section{Concluding remarks: Generalizations and open questions}\label{sec:concluding}

Let us start off by remarking that the bound on the number of iterated projections and normal transformations may not be optimal, and the reason for this may lie in the fact that we are not allowing the full
action by projective transformations and
normal transformations, as we want to stay in the world of polytopes. Indeed, purely from a naive dimension count for the realization space of the $d$-cube $(2d^2)$ compared to the projective linear group $(d(d+2))$ it might be possible that only a constant number of these operations suffice (namely $3$, projective followed by normal followed by projective). We leave this as an open problem.

\begin{problem}
	How many normal and projective transformations are needed to transform any combinatorial cube into the standard one?
\end{problem}

Second is the natural question of more classes of combinatorial types of polytopes that are connected by normal and projective transformations. Let us call those polytopes $\pn$-unique. Dually, let us call polytopes weakly-$\pr$-unique if they are related by projective transformation, and a movement of its vertices along the rays they generate within the same combinatorial type.
But in the dual, this permits moving some facet hyperplanes to infinity.
If we want the dual to $\pn$-uniqueness, then we add the condition that the origin has to be in the interior of the polytope at all times; we call such polytopes $\pr$-unique. Then the $\pr$-unique polytopes are precisely dual to the $\pn$-unique polytopes.

We note the following simple fact about
simplicial
stacking (connected sum with a
simplex $S$) on $\pr$-unique polytopes:
\begin{proposition}
	If $P$ is $\pr$-unique and $F$ a simplex facet of $P$, then $P\#_F S$ is $\pr$-unique.
\end{proposition}
\begin{proof}
Do $\pr$-transformations so that the $P$ part has the correct shape, then get the new vertex $v$ to its desired position $u$ with transformations that do not effect the $P$ part: this can always be done with a sequence of $3$ $\pr$-transformations. For example, 
scale $v$ by $\epsilon$ so that $\epsilon v$ is close enough to $F$, namely so that the line through $u$ and $\epsilon v$ intersects the interior of $F$, say at $w$. Then move the origin to $w$, then scale $\epsilon v$ to $u$.
\end{proof}

We immediately conclude:

\begin{corollary}
Every polygon, and more generally every stacked polytope, is $\pr$-unique. In particular, in every dimension $d\geq 2$, there are infinitely many combinatorial types of $\pr$-unique polytopes.
\end{corollary}

This is in contrast to projectively unique polytopes, which are only finitely many in dimension $2$ and $3$. (However, in sufficiently large fixed dimension $d$ there exist projectively unique $d$-polytopes with arbitrarily many vertices -- this was proved for $d\ge 69$ by Adiprasito and Ziegler~\cite{AdipZ15}, answering a question of Perles and Shephard~\cite{PerlS74}.)

We use the notation of~\cite{McMu76} about free joins and subdirect sums, and note that the following can be said:

\begin{proposition}
	The free join of two polytopes $P$ and $Q$ is weakly-$\pr$-unique if and only if both components are.
\end{proposition}

This follows easily, as we may act on each component separately.
The same is not true for $\pr$-uniqueness, and therefore $\pn$-uniqueness. A counterexample is the cone over the crosspolytope.
Indeed, it follows from the following observation, that is straightforward from the definitions:
\begin{lemma}
If $P$ is $\pr$-unique then every facet of $P$ is projectively unique.
\end{lemma}

The next result holds for $\pn$-uniqueness, by following the proof of Proposition~\ref{prop:main} for the cube case.

\begin{theorem}
	The subdirect sum of a $\pr$-unique polytope with a simplex is $\pr$-unique, and vice versa. Dually, the subdirect product with a simplex is $\pn$-unique if and only if the original polytope is.
\end{theorem}

This is especially interesting if one considers only those polytopes that are obtained as products of simplices. These are $\pn$-unique by the above theorem (and include the cube). Moreover, the algorithm described in Proposition~\ref{prop:main} goes through verbatim, and is continuously dependent on the starting geometry. Hence, we once again obtain that the realization space of such polytopes is contractible (a fact not known for general $\pn$-unique or $\pr$-unique polytopes). We end with a question:

\begin{problem}
	Is the dodecahedron $\pn$-unique?
\end{problem}

\subsection*{Acknowledgements}
We thank Ron Adin for helpful discussions.
The first author acknowledges support by ERC StG 716424 - CASe and ISF Grant 1050/16.
The second and third authors acknowledge support by ISF grant 1695/15 and by ISF-BSF joint grant 2016288.


\bibliographystyle{myamsalpha}
\bibliography{cubical}

\end{document}